\documentclass[12pt]{article}
\usepackage{graphicx} 
\usepackage{amsmath}
\usepackage{amsthm}
\usepackage[a4paper, total={6in, 9in}]{geometry}
\usepackage{amsfonts}
\usepackage{amssymb}
\usepackage{graphicx}
\usepackage{enumitem}
\newtheorem{theorem}{Theorem}[section]

\newtheorem{lemma}{Lemma}[section]

\theoremstyle{definition}
\newtheorem{definition}{Definition}
\newtheorem{remark}{Remark}

\numberwithin{equation}{section}

\begin{document}

\long\def\symbolfootnote[#1]#2{\begingroup%
\def\thefootnote{\fnsymbol{footnote}}\footnote[#1]{#2}\endgroup}
\title{New Characterizations of $\mathcal{N}(p,q,s)$ spaces on the unit ball of $\mathbb{C}^n$.}
\author{Athanasios Beslikas}
\maketitle

\begin{abstract}
In this note  we provide Holland-Walsh-type characterizations for functions on the $\mathcal{N}(p,q,s)$ spaces on the unit ball for specific values of $p\ge 1.$
Characterizations for the holomorphic function spaces $\mathcal{N}(p,q,s)$
were studied extensively in \cite{Hu} and \cite{Li}. 
 \end{abstract}
 
\maketitle
\section{Introduction and notation} We denote by $z=(z_1,...,z_n), z\in \mathbb{C}^n$ and $|z|^2=|z_1|^2+...+|z_n|^2=\langle z,z\rangle,$
where $\langle.,.\rangle$ denotes the usual inner product that induces the Euclidean norm in $\mathbb{C}^n.$
Let $\mathbb{B}$ denote the open unit ball in $\mathbb{C}^n$, that is $\mathbb{B}=\{z\in \mathbb{C}^n:|z|<1\}$. The class of holomorphic functions on the unit ball will be denoted by $\mathrm{Hol}(\mathbb{B}).$ We denote by $dV(z)$ the usual volume measure normalized over the unit ball, and with
$$dV_{\alpha}(z)=\frac{\Gamma(n+\alpha+1)}{n!\Gamma(\alpha+1)}(1-|z|^2)^{\alpha}dV(z),\alpha\ge -1$$ where $\Gamma(\cdot)$ is the usual Gamma function. For the whole length of our paper we denote by $\Phi_a(z)$ the involutive automorphisms of the unit ball. We have an explicit representation of such functions for $a\in\mathbb{B}\setminus\{0\},$ given by the following formula:
$$\Phi_a(z)=\frac{a-P_az-s_aQ_az}{1-\langle z,a\rangle},z\in\mathbb{B},$$
where $s_a=\sqrt{1-|a|^2}P_a$, $P_a=\frac{\langle z,a\rangle}{|a|^2}$ is the orthogonal projection of $z$ onto the space spanned by $a,$ and $Q_a=I-P_a.$ For more details the reader can check \cite{Rudin}. Now we define the M\"obius invariant measure by
$$d\lambda (z)=(1-|z|^2)^{-n-1}dV(z).$$ It is called M\"obius invariant because of the following equality:
$$\int_{\mathbb{B}}f(z)d\lambda(z)=\int_{\mathbb{B}}f\circ \Phi_a(z)d\lambda(z).$$
Also, by $\nabla f$ we symbolise the complex gradient of a holomorphic function $f$, that is $$\nabla f(z)=\left(\frac{\partial f}{\partial z_1},...,\frac{\partial f}{\partial z_n}\right)$$
and by $\widetilde{\nabla}f$ the M\"obius invariant gradient, that is: $$\widetilde{\nabla} f(z)=\nabla (f\circ \Phi_z)(0).$$ Furthermore, by $\mathcal{R}f(z)$ we denote the radial derivative of $f,$ that is:
$$\mathcal{R}f(z)=\sum_{k=1}^{n}z_k\frac{\partial f}{\partial z_k}.$$
Lastly, whenever we encounter the notation $a\asymp b,$ we simply mean that there exist two positive constants $C_1,C_2$ such that $C_1a\leq b\leq C_2a.$ 
\section{Definitions and tools} In this section we give the definition of the $\mathcal{N}(p,q,s)$ spaces, and some known characterizations. These will constitute our main tools for proving our results.
\begin{definition} \textit{Let $f\in \mathrm{Hol}(\mathbb{B}).$ For $p\ge 1, q>0,s>0$ we define the holomorphic function spaces $\mathcal{N}(p,q,s),$ or $\mathcal{N}_{q,s}^{p}$ by:}
$$\mathcal{N}_{q,s}^{p}=\Biggl\{f\in \mathrm{Hol}(\mathbb{B}):\sup_{a\in\mathbb{B}}\int_{\mathbb{B}}|f(z)|^p(1-|z|^2)^q(1-|\Phi_a(z)|^{2})^{ns}d\lambda(z)<+\infty \Biggr\}.$$
\end{definition}
For different values of the parameters $p,q,s$ we obtain different function spaces. The reader can check \cite{Hu} for such examples and more details about the $\mathcal{N}(p,q,s)$ spaces. For instance, if we could allow $\mathcal{N}(2,0,s)$, $\frac{n-1}{n}<s<1$ we would receive the classical $Q_s$ spaces, and if the parameter $s$ satisfied $1<s<\frac{n}{n-1},$ then we would receive the Bloch space $\mathcal{B}$ on the unit ball. By definition, we are not allowed to do so, and this shows that the theory of $Q_s$ spaces is completely independent to the one of $\mathcal{N}(p,q,s)$ spaces. Some basic information about the mentioned function spaces is that we already know that for $p\ge 1,q,s>0$ the set of polynomials is dense in $\mathcal{N}(p,q,s),$ if and only if $ns+q>n.$ Also for $p\ge 1,q,s>0$ we know that these spaces are functional Banach spaces. For the purpose of our work, we will now mention some known characterizations. 
\begin{theorem} \textit{Let $f\in \mathrm{Hol}(\mathbb{B})$ and $p\ge 1, q>0$ and $s>\max \{0,1-\frac{q}{n}\}.$ Then $f\in \mathcal{N}(p,q,s)$ if and only if:}
\begin{align}
 I_1=&\sup_{a\in \mathbb{B}}\int_{\mathbb{B}}|\nabla f(z)|^p(1-|z|^2)^{p+q}(1-|\Phi_a(z)|^2)^{ns}d\lambda(z)<+\infty.&\\
 I_2=&\sup_{a\in \mathbb{B}}\int_{\mathbb{B}}|\widetilde{\nabla} f(z)|^p(1-|z|^2)^{q}(1-|\Phi_a(z)|^2)^{ns}d\lambda(z)<+\infty.&\\
 I_3=&\sup_{a\in \mathbb{B}}\int_{\mathbb{B}}|\mathcal{R}f(z)|^p(1-|z|^2)^{p+q}(1-|\Phi_a(z)|^2)^{ns}d\lambda(z)<+\infty.&
\end{align}
\end{theorem}
\begin{proof} The proofs for all of the 3 characterizations can be found in \cite{Li}.
\end{proof}
The next Theorem that we will state is of particular interest for us.
\begin{theorem}\textit{Suppose $f\in \mathrm{Hol}(\mathbb{B}) ,p\ge 1,q>0$ and also $s>\max\{0,1-\frac{q}{n}\}, \alpha>q+ns-n-1.$ Then $f\in\mathcal{N}(p,q,s)$ if and only if:}
$$\sup_{a\in\mathbb{B}}\int_{\mathbb{B}}\int_{\mathbb{B}}\frac{|f(z)-f(w)|^p}{|1-\langle z, w \rangle|^{2(n+1+\alpha)}}(1-|z|^2)^{q}(1-|\Phi_{a}(z)|^2)^{ns}dV_{\alpha}(z)dV_{\alpha}(w)<+\infty$$
\end{theorem}
\begin{proof} See Theorem 5.12 of \cite{Li}.
\end{proof}
\section{Motivation and statement of results}
 Historically, the problem started from \cite{Holland} and the Holland-Walsh characterization of the classical and well studied Bloch space $\mathcal{B}$ on the unit disc. Later, in \cite{Stroethoff}, Stroethoff, in a much simpler and elegant way, gave the Holland-Walsh characterization for the Bloch space and also provided a similar characterization for the Besov spaces $B_p$ on the unit disc $\mathbb{D},$ for $p\ge 2.$   After defining and studying extensively the Hardy, Bergman, Besov and Dirichlet-type spaces on the unit ball of $\mathbb{C}^n$, some similar double integral characterizations emerged for the cases of Dirichlet, Bergman, Besov, and Bloch-type spaces, and this was followed up by similar characterizations for the $Q_s$ spaces. Later, the authors of \cite{Li} treated partially the case of $\mathcal{N}(p,q,s)$ spaces. In the present note, we provide three characterizations that have not been found in \cite{Li}, that are similar to the ones presented in \cite{Lii}, \cite{Michalska}, \cite{Pavlovic} for specific values of the parameter $p.$ Let us now state the main results.
\begin{theorem}\textit{Let $f\in \mathrm{Hol}(\mathbb{B}) ,q>0$. Suppose also that the parameters $p,q,s$ satisfy: $\alpha>q+ns-n-1,$ $s>\max\{0,1-\frac{q}{n}\}$ and $p\ge2(n+1+\alpha).$ Then, $f\in \mathcal{N}(p,q,s),$ if and only if:}
$$\sup_{a\in\mathbb{B}}\int_{\mathbb{B}}\int_{\mathbb{B}}\frac{|f(z)-f(w)|^p}{|z-w|^{2(n+1+\alpha)}}(1-|z|^2)^{q}(1-|\Phi_{a}(z)|^2)^{ns}dV_{\alpha}(w)dV_{\alpha}(z)<+\infty.$$
\end{theorem}  
\begin{remark} Actually, because of the fact that $\alpha>q+ns-n-1$ the Theorem can hold for $p>2(q+ns).$ 
\end{remark}
Our second Theorem is quite similar to the first one:
\begin{theorem}\textit{Let $f\in \mathrm{Hol}(\mathbb{B}) ,q>0$. Suppose also that the parameters $p,q,s$ satisfy: $\alpha>q+ns-n-1,$ $s>\max\{0,1-\frac{q}{n}\}$ and $p\ge2(n+1+\alpha).$ Then, $f\in \mathcal{N}(p,q,s),$ if and only if:}
\begin{multline}\sup_{a\in\mathbb{B}}\int_{\mathbb{B}}\int_{\mathbb{B}}\frac{|f(z)-f(w)|^p}{|w-P_w z-s_wQ_w z|^{2(n+1+\alpha)}}(1-|z|^2)^{q}(1-|\Phi_{a}(z)|^2)^{ns}dV_{\alpha}(w)dV_{\alpha}(z)<+\infty.\end{multline}
\end{theorem}  
Our third result is straightforward and can be deduced very easily by Theorem 2.3. The motive behind it can be found in \cite{Lii} again.
\begin{definition} \textit{Let $f\in \mathrm{Hol}(\mathbb{B}).$ We define the $p$-Mean Oscillation of $f$ as follows:}
$$MO_{p}(f)(z)=\left(\int_{\mathbb{B}}|f(z)-f(w)|^p\frac{(1-|z|^2)^{n+1}}{|1-\langle z,w \rangle|^{2(n+1)}}dV(w)\right)^{1/p}.$$
\end{definition}
By the previous definition, we get:
\begin{theorem} \textit{Let $f\in \mathrm{Hol}(\mathbb{B}).$ If $p\ge 1, q>0$ and $s>\max \{0,1-\frac{q}{n}\},$ then $f\in N(p,q,s)$ if and only if:}
$$J(f)=\sup_{a\in \mathbb{B}}\int_{\mathbb{B}} MO^p_{p}(f)(z)(1-|z|^2)^q(1-|\Phi_a(z)|^2)^{ns}d\lambda(z)<+\infty.$$
\end{theorem} 
\section{Some Lemmata and proofs}
Initially, we provide the reader with some lemmata that will constitute our main tools for the proof of Theorem 3.1. Let $\Phi_z(w),z,w\in\mathbb{B},z\neq w,$ be an involutive automorphism of the unit ball on $\mathbb{C}^n.$ As in \cite{Li},\cite{Lii} and \cite{Liii}, we consider the \textit{Bergman-pseudometric} on the unit ball as:
$$d(z,w)=|\Phi_z(w)|.$$ We denote by$$E(z,r)=\{w\in \mathbb{B}:|\Phi_z(w)|<r\}$$
the \textit{Bergman-metric ball,} centered at $z$ with radius $r>0.$ The following well known Lemma holds:
\begin{lemma} \textit{For any $r>0$ and $z\in \mathbb{B}$, let $E(z,r)$ be the Bergman metric ball centered at $z.$ Then:}
$$(1-|z|^2)\asymp(1-|w|^2)\asymp|1-\langle z,w\rangle|$$
\textit{for all $z \in \mathbb{B}, w\in E(z,r).$}
\end{lemma}
\begin{proof}
The proof can be found in \cite{Zhu}.
\end{proof}
\begin{lemma} \textit{Let $f\in \mathrm{Hol}(\mathbb{B}).$ If $f(0)=0$ then, for all $p\ge 2(n+1+\alpha)$, $\alpha\ge -1,$ there exists a positive constant $C>0$ such that:}
$$\int_{\mathbb{B}}\frac{|f(w)|^p}{|w|^{2(n+1+a)}}dV_a(w)\leq C\int_{\mathbb{B}}|f(w)|^pdV_{\alpha}(w).$$
\end{lemma}
\begin{proof} The Lemma holds for $p=2(n+1+\alpha)$ (see lemma 2.2. of \cite{Lii}). So for all $w\in \mathbb{B},$ we have the trivial inequality:
$$\frac{1}{|w|^{2(n+1+\alpha)}}\leq \frac{1}{|w|^{p}},$$
for all $p\ge 2(n+1+\alpha).$
\end{proof}
\begin{lemma} \textit{Let $z,w \in \mathbb{B}, z\neq w.$ Then, the following inequalities are true:}
\begin{equation}
|z-\Phi_z(w)|\ge \frac{|w|(1-|z|^2)}{|1-\langle z,w \rangle |}
\end{equation}
\begin{equation}
|z-\Phi_z(w)|^2\leq \frac{2(1-|z|^2)}{|1-\langle z,w \rangle|}
\end{equation}
\end{lemma}
\begin{proof} The proof can be found again in \cite{Zhu}.
\end{proof} 
We are ready to proceed with the proof of Theorem 3.1.
\begin{proof} (Proof of Theorem 3.1.): \textit{Sufficiency}: Let $f\in \mathrm{Hol}(\mathbb{B}),$ and $p,q,s,\alpha$ as in the statement of the Theorem. Initially, for the convenience of the reader we set:
\begin{equation}
I_a(f)=\int_{\mathbb{B}}\int_{\mathbb{B}}\frac{|f(z)-f(w)|^p}{|z-w|^{2(n+1+\alpha)}}(1-|z|^2)^{q}(1-|\Phi_{a}(z)|^2)^{ns}dV_{\alpha}(w)dV_{\alpha}(z)
\end{equation}
We assume firstly that $\sup_{a\in\mathbb{B}}I_a(f)<\infty.$ We have to prove that $f\in \mathcal{N}(p,q,s).$ For the upcoming calculations, let 
$$k_z(w)=\frac{(1-|z|^2)^{n+1+\alpha}}{|1-\langle z, w\rangle|^{2(n+1+\alpha)}}, z,w\in\mathbb{B}.$$
We will estimate $I_a(f)$ from below. To do so, we apply firstly a change of variables $w=\Phi_z(w):$
\begin{multline}
    I_a(f)=\int_{\mathbb{B}}\int_{\mathbb{B}}\frac{|f(z)-f\circ \Phi_z(w)|^p}{|z-\Phi_z(w)|^{2(n+1+\alpha)}}\times\\(1-|z|^2)^q(1-|\Phi_a(z)|^2)^{ns}k_z(w)dV_{\alpha}(z)dV_{\alpha}(w)
    \end{multline}
For convenience in the upcoming calculations, set $F_z(w)=f(z)-f\circ \Phi_z(w),$ and $n+1+\alpha=\gamma.$
By applying Fubini's Theorem, and Lemma 4.3 (inequality (4.2)), we get:
\begin{multline}
I_a(f)\ge C\int_{\mathbb{B}}(1-|z|^2)^{q}(1-|\Phi_a(z)|^2)^{ns}dV_a(z) \times \\ \left(\int_{\mathbb{B}}\frac{|F_z(w)|^p}{\frac{(1-|z|^2)^{n+1+\alpha}}{|1-\langle z,w \rangle|^{n+1+\alpha}}}\frac{(1-|z|^2)^{n+1+\alpha}}{|1-\langle z,w \rangle|^{2(n+1+\alpha)}}dV_{\alpha}(w)\right).
\end{multline}
Fix $0<r<1,$ and take $E(z,r)\subset{\mathbb{B}}$ be the domain of the second integral:   
\begin{align}
I_a(f)\ge&C\int_{\mathbb{B}}(1-|z|^2)^q(1-|\Phi_a(z)|^2)^{ns}dV_a(z)\left(\int_{\mathbb{B}}\frac{|F_z(w)|^p}{|1-\langle z,w\rangle |^{\gamma}}dV_{\alpha}(w)\right)&&\\
\ge&C\int_{\mathbb{B}}(1-|z|^2)^q(1-|\Phi_a(z)|^2)^{ns}dV_a(z)\left(\int_{E(z,r)}\frac{|F_z(w)|^p}{|1-\langle z,w\rangle |^{\gamma}}dV_{\alpha}(w)\right)&&\\
\ge&C\int_{\mathbb{B}}(1-|z|^2)^q(1-|\Phi_a(z)|^2)^{ns}dV_a(z)\left(\int_{E(z,r)}\frac{|F_z(w)|^p}{(1-|z|^2)^{\gamma}}dV_{\alpha}(w)\right).&&
\end{align}
At this point, we move $(1-|z|^2)^{-\gamma}$ to the outer integral. It is now obvious that we will replace $(1-|z|^2)^{-n-1-a}dV_a(z)$ with $d\lambda(z):$
\begin{align}
I_a(f)\ge C\int_{\mathbb{B}}(1-|z|^2)^q(1-|\Phi_a(z)|^2)^{ns}d\lambda(z) \left(\int_{E(z,r)}|F_z(w)|^pdV_{\alpha}(w)\right).&&
\end{align}
By the topology induced by the Bergman pseudometric, we can find $\rho<r$ and a Euclidean ball $B(z,\rho)\subset E(z,r)$. By applying the plurisubharmonicity property for the function $|F_z(w)|^p,p\ge 1$ we deduce:
\begin{align}
I_a(f)  &\ge C\int_{\mathbb{B}}|F_z(z)|^p(1-|z|^2)^{q}(1-|\Phi_a(z)|^2)^{ns}d\lambda(z)&&\\
        &=C\int_{\mathbb{B}}|f(z)-f(0)|^p(1-|z|^2)^{q}(1-|\Phi_a(z)|^2)^{ns}d\lambda(z).
\end{align}
By taking the supremum for $a\in\mathbb{B}$ in (4.10) we proved:
\begin{align}
+\infty>\sup_{a\in\mathbb{B}}I_a(f) \ge&\sup_{a\in\mathbb{B}}\int_{\mathbb{B}}|f(z)|^{p}(1-|z|^2)^q(1-|\Phi_a(z)|^2)^{ns}d\lambda(z)=||f||_{\mathcal{N}},
\end{align}
where $||f||_{\mathcal{N}}$ denotes the pseudo-norm induced by the definition of the $\mathcal{N}(p,q,s)$ spaces.\\\\
\textit{Necessity:} We assume that $f\in \mathcal{N}(p,q,s).$ We have to show that $\sup_{a\in\mathbb{B}}I_a(f)$ is bounded. We recall that$$F_z(w)=f(z)-f\circ\Phi_z(w), z,w\in\mathbb{B},$$ and we observe that $F_z(0)=0.$ We begin by applying Fubini's Theorem and a change of variables $w=\Phi_z(w),$ as before:
\begin{multline}
I_a(f)=\int_{\mathbb{B}}(1-|z|^2)^{q}(1-|\Phi_a(z)|^2)^{ns}dV_a(z)\times\\ \left(\int_{\mathbb{B}}\frac{|F_z(w)|^p}{|z-\Phi_z(w)|^{2(n+1+\alpha)}}\frac{(1-|z|^2)^{n+1+\alpha}}{|1-\langle z,w \rangle|^{2(n+1+\alpha)}}dV_{\alpha}(w)\right).
\end{multline}
Now we apply the inequality (4.1) and we get:
\begin{multline}
I_a(f)\leq \int_{\mathbb{B}}(1-|z|^2)^{q}(1-|\Phi_a(z)|^2)^{ns}dV_a(z)\times\\ \left(\int_{\mathbb{B}}\frac{|F_z(w)|^p}{\frac{|w|^{2(n+1+\alpha)}(1-|z|^2)^{2(n+1+\alpha)}}{|1-\langle z,w \rangle|^{2(n+1+\alpha)}}}\frac{(1-|z|^2)^{n+1+\alpha}}{|1-\langle z,w \rangle|^{2(n+1+\alpha)}}dV_{\alpha}(w)\right).
\end{multline}
After we simplify the equal terms in the numerator and denominator, we move the term $(1-|z|^2)^{-(n+1+\alpha)}$ to the outer integral, and we obtain
$$I_a(f)\leq \int_{\mathbb{B}}(1-|z|^2)^{q-(n+1+\alpha)}(1-|\Phi_a(z)|^2)^{ns}dV_a(z)\int_{\mathbb{B}}
\frac{|F_z(w)|^p}{|w|^{2(n+1+\alpha)}}dV_{\alpha}(w).$$
Now we apply Lemma 4.2 for the function $F_z(w).$ After another change of variables we deduce:
\begin{align}
I_a(f)\leq &C\int_{\mathbb{B}}(1-|z|^2)^{q-(n+1+\alpha)}(1-|\Phi_a(z)|^2)^{ns}dV_{\alpha}(z)\int_{\mathbb{B}}|F_z(w)|^pdV_a(w)&\\
    =& C\int_{\mathbb{B}}(1-|z|^2)^{q}(1-|\Phi_a(z)|^2)^{ns}dV_{\alpha}(z)\int_{\mathbb{B}}\frac{|f(z)-f(w)|^p}{|1-\langle z,w \rangle|^{2(n+1+a)}}dV_a(w)&\\
    =&C\int_{\mathbb{B}}\int_{\mathbb{B}}\frac{|f(z)-f(w)|^p}{|1-\langle z ,w\rangle|^{2(n+1+\alpha)}}(1-|z|^2)^q(1-|\Phi_a(z)|^2)^{ns}dV_{\alpha}(z)dV_{\alpha}(w),
\end{align}
hence $I_a(f)<+\infty,$ by Theorem 2.3.   
\end{proof}
We proceed now with the proof of Theorem 3.2.
\begin{proof} (Proof of Theorem 3.2) (\textit{Sufficiency}:) Suppose that:
\begin{multline}
\sup_{a\in\mathbb{B}}\int_{\mathbb{B}}\int_{\mathbb{B}}\frac{|f(z)-f(w)|^p}{|w-P_w z-s_wQ_w z|^{2(n+1+\alpha)}}\times \\(1-|z|^2)^{q}(1-|\Phi_{a}(z)|^2)^{ns}dV_{\alpha}(w)dV_{\alpha}(z)<+\infty.
\end{multline}
Then, by the trivial inequality:
$$\frac{1}{|1-\langle z,w \rangle|}\leq \frac{1}{|w-P_w z-s_wQ_w z|}$$ and Theorem 2.3 we deduce that $f\in \mathcal{N}(p,q,s).$\\\\
(\textit{Necessity}:) Let $f\in \mathcal{N}(p,q,s).$  We will estimate from above the following integral:
\begin{multline}
J_a(f)=\int_{\mathbb{B}}\int_{\mathbb{B}}\frac{|f(z)-f(w)|^p}{|w-P_w z-s_wQ_w z|^{2(n+1+\alpha)}}\times \\(1-|z|^2)^{q}(1-|\Phi_{a}(z)|^2)^{ns}dV_{\alpha}(w)dV_{\alpha}(z).
\end{multline}
Initially, we observe:
\begin{multline}
J_a(f)=\int_{\mathbb{B}}\int_{\mathbb{B}}\frac{|f(z)-f(w)|^p}{|\Phi_z(w)|^{2(n+1+a)}|1-\langle z,w\rangle|^{2(n+1+\alpha)}}\times \\(1-|z|^2)^{q}(1-|\Phi_{a}(z)|^2)^{ns}dV_{\alpha}(w)dV_{\alpha}(z).
\end{multline}
We apply the change of variables $w=\Phi_z(w),$ and Fubini's Theorem as in the proof of the previous Theorem:
\begin{multline}
I_a(f)=\int_{\mathbb{B}}(1-|z|^2)^q(1-|\Phi_a(z)|^2)^{ns}dV_{\alpha}(z)\times \\ \left(\int_{\mathbb{B}}\frac{|f\circ\Phi_z(w)-f\circ\Phi_z(0)|^pk_z(w)}{|w|^{2(n+1+\alpha)}|1-\langle \Phi_z(w),z \rangle |^{2(n+1+\alpha)}}dV_{\alpha}(w)\right)
\end{multline}
By the properties of the involutive automorphisms of the unit ball, we know that the following identity holds (e.g. [8].):
\begin{equation}
\frac{1}{|1-\langle \Phi_z(w),z \rangle|^{2(n+1+\alpha)}}=\frac{|1-\langle z,w\rangle|^{2(n+1+\alpha)}}{(1-|z|^2)^{2(n+1+\alpha)}}.\end{equation}
By applying (4.22) to (4.21) we obtain:
\begin{multline}
J_a(f)=\int_{\mathbb{B}}(1-|z|^2)^q(1-|\Phi_a(z)|^2)^{ns}d\lambda(z)\times \\ \left(\int_{\mathbb{B}}\frac{|f\circ\Phi_z(w)-f\circ\Phi_z(0)|^p}{|w|^{2(n+1+\alpha)}}dV_{\alpha}(w)\right).
\end{multline}
Recall, once again, the previous notation $F_z(w)=f\circ\Phi_z(w)-f\circ\Phi_z(0).$ At this stage, we will apply Lemma 4.2. By doing so, we obtain a positive constant $C>0$ such that:
\begin{align}
J_a(f)\leq C\int_{\mathbb{B}}(1-|z|^2)^q(1-|\Phi_a(z)|^2)^{ns}d\lambda(z) \left(\int_{\mathbb{B}}|F_z(w)|^pdV_{\alpha}(w)\right).
\end{align}
By repeating the change of variables and applying Theorem 2.3, we obtain:
\begin{multline}
\sup_{a\in\mathbb{B}}J_a(f) \leq\sup_{a\in\mathbb{B}}\int_{\mathbb{B}}\int_{\mathbb{B}}\frac{|f(z)-f(w)|^p}{|1-\langle z, w \rangle|^{2(n+1+\alpha)}}\times \\(1-|z|^2)^{q}(1-|\Phi_{a}(z)|^2)^{ns}dV_{\alpha}(z)dV_{\alpha}(w)<+\infty\end{multline}
and the second implication follows.
\end{proof}
\begin{proof} (Proof of Theorem 3.3:) Observe that: 
$$MO_p^{p}(f)=\int_{\mathbb{B}}\frac{|f(z)-f(w)|^p}{|1-\langle z,w\rangle|^{2(n+1)}}(1-|z|^2)^{n+1}dV(w).$$
From Theorem 2.3. and by choosing $\alpha=0,$ we know that $f\in N(p,q,s) $ if and only if
\begin{multline}
J(f)=\sup_{a\in\mathbb{B}}\int_{\mathbb{B}}\int_{\mathbb{B}}\frac{|f(z)-f(w)|^p}{|1-\langle z, w \rangle|^{2(n+1)}}\times\\(1-|z|^2)^{q}(1-|\Phi_{a}(z)|^2)^{ns}dV(z)dV(w)<+\infty.
\end{multline}
We write the integral of $J(f)$ as an iterated one and we observe that $f\in N(p,q,s)$ if and only if
\begin{multline}
\sup_{a\in \mathbb{B}}\int_{\mathbb{B}}\left(\int_{\mathbb{B}}\frac{|f(z)-f(w)|^p}{|1-\langle z,w\rangle|^{2(n+1)}}(1-|z|^2)^{n+1}dV(w)\right)\times\\(1-|z|^2)^q(1-|\Phi_a(z)|^2)^{ns}d\lambda(z)<+\infty
\end{multline}
which gives us the desired result.
\end{proof}
\section{Discussion and questions} In this final section we pose some comments and questions for the interested reader. The motivation behind our question is clear. We proved Theorem 3.1 and Theorem 3.2. for $p\ge2(n+1+\alpha),\alpha \ge1.$ \\\\
\textbf{Question 1:} \textit{Can we obtain Theorem 3.1 for more general values of the parameter $p\ge 1?$ What happens for $ 1\leq p<2(n+1+\alpha)?$}\\\\
The answer to this question seems to be technical, as in this case we cannot apply Lemma 4.2 as in our proofs.

Jagiellonian University\\
Institute of Mathematics,\\ \L{}ojasiewicza 6\\
30-348 Krak\'ow, Poland\\
athanasios.beslikas@doctoral.uj.edu.pl

\end{document}